\documentclass[11pt,a4paper]{article}
\usepackage{amssymb,amsmath,mathrsfs,enumerate,comment}
\usepackage{times,theorem,latexsym,color}
\allowdisplaybreaks[1]
\numberwithin{equation}{section}

\usepackage[colorlinks=true, pdfstartview=FitV, linkcolor=blue, citecolor=blue, urlcolor=blue,pagebackref=false]{hyperref}

\usepackage{tikz}
\usepackage{float}
\usepackage[font=footnotesize]{caption}

\newcommand{\BOX}{\ensuremath\Box}

\newtheorem{theorem}{Theorem}[section]

\newtheorem{corollary}[theorem]{Corollary}

{\theorembodyfont{\rmfamily}}
{\theorembodyfont{\rmfamily}}
{\theorembodyfont{\rmfamily}}

\newtheorem{proposition}[theorem]{Proposition}
{\theorembodyfont{\rmfamily}\newtheorem{remark}[theorem]{Remark}}
{\theorembodyfont{\rmfamily}}

\newcommand{\Z}{\mathbb{Z}}
\newcommand{\R}{\mathbb{R}}

\newcommand{\dd}{\,{\rm d}}
\newcommand{\opdiv}{\operatorname{div}}

\newcommand{\oprot}{\operatorname{rot}}

\newcommand{\h}{{\rm h}}
\newcommand{\vv}{{\rm v}}

\DeclareMathOperator*{\esssup}{ess\,sup}

\def\XXint#1#2#3{{\setbox0=\hbox{$#1{#2#3}{\int}$}
		\vcenter{\hbox{$#2#3$}}\kern-.5\wd0}}

\newenvironment{proof}{{\vskip\baselineskip\noindent\textbf{Proof:}}}%
{\hspace*{.1pt}\hspace*{\fill}\BOX\vskip\baselineskip}

\newenvironment{proofx}[1]%
{\vskip\baselineskip\noindent\textbf{Proof of {#1}:}}%
{\hspace*{.1pt}\hspace*{\fill}\BOX\vskip\baselineskip}
{\vskip\baselineskip\noindent\textbf{Proof of Theorem \protect\ref{#1}:}}%
{\hspace*{.1pt}\hspace*{\fill}\BOX\vskip\baselineskip}
{\vskip\baselineskip\noindent\textbf{Proof of Theorems \protect\ref{#1} --
		\protect\ref{#2}:}}%
{\hspace*{.1pt}\hspace*{\fill}\BOX\vskip\baselineskip}

\begin{document}

\title{Existence of steady Navier-Stokes flows \\ exterior to an infinite cylinder}

\author{
Mitsuo Higaki
\thanks{
Department of Mathematics, 
Graduate School of Science, 
Kobe University, 
1-1 Rokkodai, Nada-ku, Kobe 657-8501, Japan.
\textit{E-mail address:}\texttt{higaki@math.kobe-u.ac.jp}
}
\and
Ryoma Horiuchi
\thanks{
Department of Mathematics, 
Graduate School of Science, 
Kobe University, 
1-1 Rokkodai, Nada-ku, Kobe 657-8501, Japan.
\textit{E-mail address:}\texttt{ryoh0798@gmail.com}
}
}

\date{}

\maketitle

\begin{abstract}
We consider the three-dimensional steady Navier-Stokes system in the exterior of an infinite cylinder under the action of an external force. We construct solutions in the class of vertically uniform flows which vanish at horizontal infinity. More precisely, for a boundary datum determined by a rotating flow and a suction flow, and for a small force of the form $f=g+\opdiv F$ with suitable decay, we prove the existence of a weak solution asymptotic to the corresponding Hamel-type flow. Although all data are independent of the vertical variable, the problem is not reduced to the planar exterior Navier-Stokes system: the vertical component satisfies a separate transport-diffusion equation involving the two-dimensional Laplacian, whose fundamental solution has logarithmic growth. The proof is based on a mode-by-mode analysis of the linearized three-dimensional problem around the Hamel-type flow and a contraction argument.

\bigskip

\end{abstract}

\tableofcontents

\section{Introduction}
\label{sec.intro}

We consider the steady Navier-Stokes system on $\Omega\times\R$ 
\begin{equation}\tag{NS}\label{intro.eq.NS}
\left\{
\begin{array}{ll}
-\Delta u + \nabla p 
= -u\cdot\nabla u + f&\mbox{in}\ \Omega\times\R \\
\opdiv u = 0&\mbox{in}\ \Omega\times\R \\
u= b &\mbox{on}\ \partial\Omega\times\R \\
u(x,z)\to0&\mbox{as}\ |x|\to\infty. 
\end{array}\right.
\end{equation}
Here $\Omega=\{x=(x_1,x_2)\in\R^2~|~|x|>1\}$ denotes an exterior disk, and thus $\Omega\times\R$ is the exterior of an infinite cylinder. The velocity field $u=(u_1,u_2,u_3)$ and the pressure field $p$ are unknown functions, while the external force $f$ and the boundary condition $b$ are given. We use standard notation for derivatives: $\partial_j = \partial/\partial x_j$, $\partial_z = \partial/\partial z$, $\Delta = \sum_{j=1}^2 \partial^2_j + \partial_z^2$, $\nabla = (\partial_1, \partial_2, \partial_z)$, $\opdiv u = \sum_{j=1}^2 \partial_j u_j + \partial_z u_3$ and $u\cdot \nabla u = (\sum_{j=1}^2 u_j \partial_j + u_3\partial_z) u$.

Our interest is to construct a solution for $f$ that does not decay in the vertical direction. This configuration includes the planar exterior problem for the steady Navier-Stokes system if one takes $f=(f_1(x),f_2(x),0)$. It is well-known that this problem has characteristic difficulties due to the lack of certain embeddings in two-dimensional unbounded domains and to the Stokes paradox \cite{CF1961, KS1992, Galdi2004, Russo2010, Galdi2011}. There are many open problems although mathematically fundamental. For an overview, see the recent survey \cite{KR2023} by Korobkov-Ren.

Because of the absence of general theory, it is an important subject to construct planar steady Navier-Stokes flows for given data $\Omega,f,b$ belonging to an appropriate class, and to study their properties at infinity. There are two approaches in this context. One is based on symmetry. We impose symmetry on the given data to improve decay of solutions by cancellation in integrals; see \cite{Galdi2004, Russo2009, Galdi2011, Yamazaki2011, PR2012, Yamazaki2016, Yamazaki2018} for the work in this direction.

The other approach is more constructive and is based on perturbation. We perturb the Navier-Stokes system around an exact solution invariant under the scaling symmetry 
\[
u(x) \mapsto \lambda u(\lambda x), 
\quad 
p(x) \mapsto \lambda^2 p(\lambda x) 
\]
and construct a solution, regarded as a remainder at infinity, to the perturbed system. The underlying idea is that the transport by a scale-invariant flow improves decay structure of solutions to the perturbed system. We typically use an exact solution found by Hamel \cite{Hamel1917}: 
\[
{\mathcal V}(x) 
=  
\alpha \frac{x^{\perp}}{|x|^2} - \gamma \frac{x}{|x|^2}, 
\quad 
{\mathcal Q}(x) = -\frac{|{\mathcal V}(x)|^2}{2}, 
\quad 
\gamma,\alpha\in\R, 
\]
where $x^\bot=(-x_2,x_1)$ and we remark that ${\mathcal V}$ is a linear combination of the rotating flow $x^{\perp}/|x|^2$ and the flow $x/|x|^2$ carrying flux $-2\pi\gamma$. Indeed, solutions around $({\mathcal V}, {\mathcal Q})$ are constructed by Hillairet-Wittwer \cite{HW2013} for $|\alpha|>\sqrt{48}$ when $\gamma=0$ and by \cite{Higaki2023} for any $\alpha$ when $\gamma>2$. As related work, we refer to \cite{GHM2019} analyzing the exterior problem around the fast rotating flows and Maekawa-Tsurumi \cite{MT2023} constructing the Navier-Stokes flows in $\R^2$ around rotating flows. Besides, Guillod-Wittwer \cite{GW2015} generalizes the Hamel solutions.

In this paper, we use the perturbation method to construct solutions around $({\mathcal V}, {\mathcal Q})$ to the three-dimensional system \eqref{intro.eq.NS}. The content of the present work is not merely to repeat the two-dimensional argument in a different notation. There are at least two genuinely three-dimensional issues. First, since the equations are three-dimensional, the vorticity-streamfunction formulation useful in two-dimensional settings \cite{HW2013, Higaki2023, MT2023} cannot be applied directly. Second, even when the data are independent of the vertical variable, the vertical component is governed by a separate scalar transport-diffusion equation. Its linearization contains the two-dimensional Laplacian, whose fundamental solution has logarithmic growth. Thus the desired spatial decay of the vertical component is not an immediate consequence of the planar theory and has to be proved independently.

Let us also mention two developments that appeared after the first version of this preprint. Guo-Hillairet \cite{GH2024} extended the planar perturbative theory around Hamel-type flows and explicitly listed the present work among the three-dimensional extensions of the Hillairet-Wittwer approach \cite{HW2013}. Higaki \cite{Higaki2025} developed another three-dimensional extension by constructing axisymmetric steady flows around an infinite cylinder under suction, allowing neither periodicity nor decay in the vertical variable. These results are complementary to the present theorem. Here we treat general Fourier modes of vertically uniform data and allow an arbitrary rotation parameter $\alpha$, whereas the axisymmetric theory in \cite{Higaki2025} focuses on suction and requires additional restrictions when rotation is added; see \cite[Appendix B]{Higaki2025}.

We now introduce notation for the main result. For $\rho\ge0$, we define 
\begin{align}\label{def.L^infty_rho}
\begin{split}
L^\infty_\rho(\Omega) = \{ f\in L^\infty(\Omega)~|~ 
\|f\|_{L^\infty_\rho} 
<\infty \}, \qquad
\|f\|_{L^\infty_\rho} := \esssup_{x\in\Omega}\, |x|^\rho |f(x)|. 
\end{split}
\end{align}
In the cylindrical coordinates on $\Omega\times\R$ 
\begin{align}\label{def.cyl.coord.}
\begin{split}
&x_1 = r\cos \theta, 
\quad
x_2 = r\sin \theta, 
\quad 
r= |x| \ge 1, \quad \theta\in [0,2\pi),\\
&
{\bf e}_r = \Big(\frac{x}{|x|}, 0\Big), 
\quad 
{\bf e}_\theta = \Big(\frac{x^\bot}{|x|}, 0\Big), 
\quad 
{\bf e}_3 = (0,0,1), 
\end{split}
\end{align}
we denote the three-dimensional vector field $v=(v_1,v_2,v_3)$ by 
\[
v 
= v_r {\bf e}_r 
+ v_\theta {\bf e}_\theta 
+ v_3 {\bf e}_3,
\qquad
(v_r, v_\theta)
:=
\big(
(v_1,v_2)\cdot {\bf e}_r, 
(v_1,v_2)\cdot {\bf e}_\theta 
\big). 
\]
For a scalar function $s=s(r,\theta,z)$ on $\Omega\times\R$ and $n\in\Z$, we set 
\[
s_n(r,z) 
= \frac{1}{2\pi} \int_{0}^{2\pi} s(r,t,z) e^{-int} \dd t. 
\]
Then, for a vector field $v$ on $\Omega\times\R$, we define the operator ${\mathcal P}_n$ by 
\[
({\mathcal P}_n v)(r,\theta,z) 
= v_{r,n}(r,z) e^{in\theta} {\bf e}_r 
+ v_{\theta,n}(r,z) e^{in\theta} {\bf e}_\theta 
+ v_{3,n}(r,z) e^{in\theta} {\bf e}_3. 
\]
We also define ${\mathcal P}_n$ acting on scalar- and tensor-valued mappings in an obvious manner.

Our main result states the existence of solutions to \eqref{intro.eq.NS} which are vertically uniform. Before stating it, we record why the threshold $\gamma>2$ is natural in the present formulation. The restriction comes from the axisymmetric zero mode of the linearized horizontal problem; see Section \ref{subsubsec.axisym.horiz}. The corresponding homogeneous solutions are $r^{-1}$ and $r^{-\gamma+1}$. In Theorem \ref{thm.main}, the coefficients of the $O(r^{-1})$ terms at infinity are fixed in advance by the same parameters $\alpha,\gamma$ that appear in the boundary datum. Hence the homogeneous solution is only allowed to behave $O(r^{-\gamma+1})$ at spatial infinity and has to be absorbed into the remainder $O(r^{-\rho+1})$, which requires $\rho\le\gamma$. Since the space for the contraction is chosen with $\rho>2$, this leads to $\gamma>2$. In this sense the assumption is tied to the prescribed normalization of, in particular, the circulation at infinity. The planar theory \cite[Theorem 1]{GH2024} admits a wider parameter range once the circulation at infinity is allowed to vary, but for $\gamma\le2$ one should in general allow the circulation parameter at infinity to differ from the one prescribed on the boundary.

%
\begin{theorem}\label{thm.main}
For $\alpha\in\R$, $\gamma>2$ and $2<\rho<3$ with $\rho\le\gamma$, there is a constant $\varepsilon=\varepsilon(\alpha,\gamma,\rho)$ such that the following holds. Suppose that the boundary data $b$ is given by 
\[
b = b(x) = (\alpha x^{\bot} - \gamma x,0)
\]
and that the external force $f$ is a distribution on $\Omega$ and given by 
\[
f = g + \opdiv F 
\]
where $g\in L^\infty_{2\rho-1}(\Omega)^3$ and $F\in L^\infty_{2(\rho-1)}(\Omega)^{3\times3}$ satisfy 
\begin{align*}
\sum_{n\in\Z} 
\big(
\|\mathcal{P}_n g\|_{L^\infty_{2\rho-1}} 
+ 
\|\mathcal{P}_n F\|_{L^\infty_{2(\rho-1)}}
\big)
\le 
\varepsilon. 
\end{align*}
Then there is a weak solution $u=u(x)$ of \eqref{intro.eq.NS} unique in a suitable set (defined in the proof in Section \ref{sec.pf}). Moreover, the solution $u(x)$ has the asymptotic behavior 
\begin{align}\label{est1.thm.main} 
u(x) 
= 
\alpha \Big(\frac{x^{\bot}}{|x|^2},0\Big)
- \gamma \Big(\frac{x}{|x|^2}, 0\Big)
+ O(|x|^{-\rho+1})
\quad 
\mbox{as}\ |x|\to\infty. 
\end{align}
\end{theorem}
%

%
\begin{remark}\label{rem.thm.main}
\begin{enumerate}[(i)]

\item\label{item1.rem.thm.main} 
The precise definition of weak solutions is given in Section \ref{subsec.sec.intro}.

\item\label{item2.rem.thm.main}
The weak formulation in Theorem \ref{thm.main} is dictated by the class of force $f=g+\opdiv F$. If $F$ is assumed only in a weighted $L^\infty$ class, then $\opdiv F$ is a distribution in general. Hence one should not expect $W^{2,2}_{{\rm loc}}$ regularity of the velocity without imposing local differentiability on $F$. Under stronger local assumptions on $g$ and $F$, standard elliptic regularity recovers the corresponding bootstrap regularity. In the special case $F=0$ and $g_3=0$, the horizontal part of Theorem \ref{thm.main} reduces to \cite[Theorem 1.1]{Higaki2023}. Indeed, in this setting the solution belongs to $W^{2,2}_{{\rm loc}}(\overline{\Omega})$ by elliptic regularity.

\item\label{item3.rem.thm.main}
Even though the given data $f,b$ are independent of the variable $z$, the solvability of \eqref{intro.eq.NS} cannot be reduced to that of the 2D Navier-Stokes system on $\Omega$. Indeed, for a smooth external force $f=f(x)$, the vertical component $u_3=u_3(x)$ is subject to 
\[
-(\partial_1^2 + \partial_2^2) u_3 
= -(u_1\partial_1+u_2\partial_2)u_3 + f_3. 
\]
After linearization, the 2D Laplace operator appears, and its fundamental solution has logarithmic growth. Thus the proof of the decay of $u_3$ requires a separate analysis. This is one of the genuinely three-dimensional points of the paper.

\item\label{item4.rem.thm.main} 
It is physically more natural to consider solutions of \eqref{intro.eq.NS} for given data that are neither vertically uniform nor decaying. However, the problem is difficult even for the vertically periodic data. Indeed, in addition to the analysis of the linearized problem in Section \ref{sec.lin.prob.}, we need to deal with the linear problems corresponding to the non-constant periodic data, which cannot be regarded as small perturbations from the Stokes system due to $\gamma>2$. New machinery is needed for this family of linear problems. Let us mention Kozono-Terasawa-Wakasugi \cite{KTW2023}, which provides asymptotic behavior for vertically periodic axisymmetric no-swirl steady flows under a generalized finite Dirichlet integral condition, and Higaki \cite{Higaki2025}, which constructs axisymmetric steady flows under suction without assuming either periodicity or decay in the vertical variable.

\end{enumerate}
\end{remark}
%

Let us outline the proof of Theorem \ref{thm.main} whose details will be given in Section \ref{sec.pf}. To adapt the above exact solution $({\mathcal V}, {\mathcal Q})$ to a three-dimensional system \eqref{intro.eq.NS}, we set
\begin{align}\label{def.V}
V(x) 
= ({\mathcal V}(x),0)
= \alpha \Big(\frac{x^{\bot}}{|x|^2},0\Big)
- \gamma \Big(\frac{x}{|x|^2}, 0\Big)
\end{align}
and $Q={\mathcal Q}$. Then the new pair $(v,q):=(u-V, p-Q)$ solves 
\begin{equation}\label{intro.eq.NP0}
\left\{
\begin{array}{ll}
-\Delta v 
+ v\cdot\nabla V 
+ V\cdot\nabla v 
+ \nabla q 
= 
- v\cdot\nabla v
+ f&\mbox{in}\ \Omega\times\R \\
\opdiv  v = 0&\mbox{in}\ \Omega\times\R \\
v= 0 &\mbox{on}\ \partial\Omega\times\R \\
v(x,z)\to0&\mbox{as}\ |x|\to\infty. 
\end{array}\right.
\end{equation}
By the formula 
\begin{align}\label{intro.eq.bilinear.rot}
u\cdot\nabla v + v\cdot\nabla u 
= -u\times\oprot v - v\times\oprot u 
+ \nabla\Big(
\frac{|u+v|^2-|u|^2-|v|^2}2
\Big), 
\end{align}
where $\times$ is the cross product in $\R^3$, and by $\oprot V=0$, we rewrite \eqref{intro.eq.NP0} as 
\begin{equation}\tag{NP}\label{intro.eq.NP}
\left\{
\begin{array}{ll}
-\Delta v -V\times\oprot v + \nabla q_1 
= 
- v\cdot\nabla v 
+ f&\mbox{in}\ \Omega\times\R \\
\opdiv   v = 0&\mbox{in}\ \Omega\times\R \\
v= 0 &\mbox{on}\ \partial\Omega\times\R \\
v(x,z)\to0&\mbox{as}\ |x|\to\infty 
\end{array}\right.
\end{equation}
with the pressure
\[
\nabla q_1 
= \nabla 
\Big(
q + \frac{|V+v|^2-|V|^2-|v|^2}2
\Big). 
\]

We analyze the perturbed nonlinear system \eqref{intro.eq.NP} to prove Theorem \ref{thm.main}. As will be seen in Section \ref{sec.lin.prob.}, the fundamental solution for the linearized problem has better decay compared with the one for the non-perturbed case when $\alpha=\gamma=0$. Consequently, an improvement of decay can be seen similar to the two-dimensional case \cite{HW2013, Higaki2023, MT2023}. However, a new analysis is necessary because the system \eqref{intro.eq.NP} is three-dimensional, and in particular because the vertical component has to be controlled separately.

This paper is organized as follows. In Section \ref{sec.lin.prob.}, we study the linearized problem of \eqref{intro.eq.NP} and prove decay estimates of the solutions. In Section \ref{sec.pf}, we prove Theorem \ref{thm.main}.

\subsection*{Notation and Terminology}
\label{subsec.sec.intro}

We summarize the notation and terminology used throughout this paper.

\smallskip

\noindent
{\bf Notation.} We denote by $C$ the constant and by $C(a,b,c,\ldots)$ the constant depending on $a,b,c,\ldots$. Both of these may vary from line to line. We use the function spaces on $\Omega$
\begin{align*}
\widehat{W}^{1,2}(\Omega)
&=\{\chi\in L^2_{{\rm loc}}(\overline{\Omega})~|~\nabla \chi\in L^2(\Omega)^2\}, \\
C^\infty_{0,\sigma}(\Omega)
&=\{\psi=(\psi_1, \psi_2)\in C^\infty_0(\Omega)^2~|~\partial_1 \psi_1 + \partial_2 \psi_2=0\},
\end{align*}
and $L^2_{\sigma}(\Omega)$ which is the completion of $C^\infty_{0,\sigma}(\Omega)$ in the $L^2$-norm. If there is no confusion, we use the same notation to denote the quantities concerning scalar-, vector- or tensor-valued mappings. For example, $\langle \cdot, \cdot\rangle$ denotes the inner product on $L^2(\Omega)$, $L^2(\Omega)^2$ or $L^2(\Omega)^{2\times2}$.

\smallskip

\noindent
\underline{{\it Weak solutions}}. 
Let us clarify the definition of weak solutions of \eqref{intro.eq.NS} 
in Theorem \ref{thm.main}. Let $f$ and $b$ satisfy the assumption in Theorem \ref{thm.main}. Then a three-dimensional vector field $u\in \widehat{W}^{1,2}(\Omega)^3$ is called a weak solution of \eqref{intro.eq.NS} if $u$ satisfies $\opdiv u=\partial_1 u_1 + \partial_2 u_2=0$ in the sense of distributions, $(u-b)|_{\partial\Omega}=0$ in the sense of trace, and it holds that 
\begin{align*}
\int_\Omega \nabla u \cdot \nabla \varphi 
= \int_\Omega (u\otimes u) \cdot \nabla \varphi 
+ \int_\Omega g \cdot \varphi 
- \int_\Omega F \cdot \nabla \varphi, &\\
\text{
for any
$\varphi
=(\varphi_1, \varphi_2, \varphi_3)
\in C^{\infty}_{0,\sigma}(\Omega)\times C^{\infty}_0(\Omega)$.}&
\end{align*}

\smallskip

\noindent
\underline{{\it Axisymmetricity}}. 
A scalar function on $\Omega\times\R$ with variable $(r,\theta,z)$ is said to be axisymmetric if it is independent of $\theta$. A vector field 
$v 
= v_r {\bf e}_r 
+ v_\theta {\bf e}_\theta 
+ v_3 {\bf e}_3$ 
on $\Omega\times\R$ with variable $(r,\theta,z)$ is said to be axisymmetric if the scalar functions $v_r, v_\theta, v_3$ are axisymmetric. The axisymmetricity of tensor fields on $\Omega\times\R$ is defined in a similar manner.

\section{Linearized Problem}
\label{sec.lin.prob.}

In this section, we consider the linearized problem of \eqref{intro.eq.NP}
\begin{equation}\tag{LP}\label{eq.LP}
\left\{
\begin{array}{ll}
-\Delta v -V\times\oprot v + \nabla q = f&\mbox{in}\ \Omega\times\R \\
\opdiv v = 0&\mbox{in}\ \Omega\times\R \\
v= 0 &\mbox{on}\ \partial\Omega\times\R \\
v(x,z)\to0&\mbox{as}\ |x|\to\infty.
\end{array}\right.
\end{equation}
Assume that this system does not depend on the variable $z$. Set 
\begin{align}\label{def.op.h.1}
v_\h=(v_1,v_2), 
\qquad 
\oprot_{{\rm 2D}} v_\h 
= \partial_1 v_2 - \partial_2 v_1. 
\end{align}
Then, by direct computation 
\begin{align*}
-V\times\oprot v 
&= 
-\left(
\begin{array}{c}
V_1 \\
V_2  \\
0 
\end{array}
\right) 
\times 
\left(
\begin{array}{c}
\partial_2 v_3 \\
-\partial_1 v_3 \\
\oprot_{{\rm 2D}} v_\h 
\end{array}
\right) \\
&= 
\left(
\begin{array}{c}
-V_2  \\
V_1  \\
0 
\end{array}
\right) 
\oprot_{{\rm 2D}} v_\h
+ \left(
\begin{array}{c}
0 \\
0 \\
V\cdot\nabla v_3
\end{array}
\right), 
\end{align*}
we separate \eqref{eq.LP} into two systems; one is for the horizontal components $v_\h$ 
\begin{equation}\tag{LP$_\h$}\label{eq.LPh}
\left\{
\begin{array}{ll}
-\Delta_\h v_\h + (V_\h)^\bot \oprot_{{\rm 2D}} v_\h + \nabla_\h q 
= f_\h&\mbox{in}\ \Omega \\
\opdiv_\h v_\h = 0&\mbox{in}\ \Omega \\
v_\h = 0 &\mbox{on}\ \partial\Omega \\
v_\h(x)\to0&\mbox{as}\ |x|\to\infty 
\end{array}\right.
\end{equation}
where 
\begin{align}\label{def.op.h.2}
-\Delta_\h=-(\partial_1^2+\partial_2^2), 
\qquad
\nabla_\h=(\partial_1,\partial_2), 
\qquad
\opdiv_\h v_\h = \partial_1 v_1 + \partial_2 v_2, 
\end{align}
and the other is for the vertical component $v_3$
\begin{equation}\tag{LP$_\vv$}\label{eq.LPv}
\left\{
\begin{array}{ll}
- \Delta_\h v_3 
+ V_\h\cdot\nabla_\h v_3
= f_3&\mbox{in}\ \Omega \\ 
v_3 = 0 &\mbox{on}\ \partial\Omega \\ 
v_3(x)\to0&\mbox{as}\ |x|\to\infty. 
\end{array}\right. 
\end{equation}

Notice that both problems are imposed on the 2D exterior domain $\Omega$. The horizontal problem \eqref{eq.LPh} is closely related to the planar exterior problem studied in \cite{Higaki2023}. Here we revisit the estimates in a weak formulation suited to forces of the form $\opdiv F$ with $F$ only in weighted $L^\infty$ spaces. On the other hand, \eqref{eq.LPv} for the vertical component is not covered by the planar theory and requires a separate consideration.

\subsection{Preliminaries}
\label{subsec.prelim}

This subsection collects notation for mappings on $\Omega\times\R$ independent of the variable $z$. We denote by $(r,\theta)$ the variables in the polar coordinates on $\Omega$; see \eqref{def.cyl.coord.}. Moreover, we identify ${\bf e}_r$ and ${\bf e}_\theta$ respectively with the two-dimensional vectors $x/|x|$ and $x^\bot/|x|$ defined on $\Omega$ if there is no confusion. Some of the notation in this subsection are special cases of those in the introduction, but are duplicated for clarity of explanation.

In the polar coordinates, we denote a two-dimensional vector field $v_\h=(v_1,v_2)$ by 
\[
v_\h 
= v_r {\bf e}_r 
+ v_\theta {\bf e}_\theta, 
\qquad
(v_r, v_\theta)
:=
\big(
(v_1,v_2)\cdot {\bf e}_r, 
(v_1,v_2)\cdot {\bf e}_\theta 
\big). 
\]
Then the operators 
$\opdiv_\h$, $\oprot_{{\rm 2D}}$ and $-\Delta_\h$ in \eqref{def.op.h.1}--\eqref{def.op.h.2} are represented by 
\begin{align}\label{formulas.polar}
\begin{split}
\opdiv_\h v_\h
& = \frac1r \partial_r (r v_r) + \frac1r \partial_\theta v_\theta, \\
\oprot_{{\rm 2D}} v_\h 
& = \frac1r \partial_r (r v_\theta) - \frac1r \partial_\theta v_r,  \\
-\Delta_\h v_\h 
& = \Big\{ -\partial_r \Big( \frac1r \partial_r (r v_r ) \Big)  
- \frac{1}{r^2} \partial_\theta^2 v_r 
+ \frac{2}{r^2} \partial_\theta v_\theta \Big\} {\bf e}_r \\
&\quad
+  \Big\{ - \partial_r \Big( \frac1r \partial_r (r v_\theta) \Big) 
- \frac{1}{r^2} \partial_\theta^2 v_\theta 
- \frac{2}{r^2} \partial_\theta v_r \Big\} {\bf e}_\theta.
\end{split}
\end{align}
If a $3\times3$ tensor field $F$ on $\Omega\times\R$ 
\begin{align*}
F 
&= 
F_{r r} {\bf e}_r\otimes {\bf e}_r 
+ F_{r \theta} {\bf e}_r\otimes {\bf e}_\theta
+ F_{r 3} {\bf e}_r\otimes {\bf e}_3 \\
&\quad
+ F_{\theta r} {\bf e}_\theta\otimes {\bf e}_r 
+ F_{\theta \theta} {\bf e}_\theta\otimes {\bf e}_\theta 
+ F_{\theta 3} {\bf e}_\theta\otimes {\bf e}_3 \\
&\quad
+ F_{3 r} {\bf e}_3\otimes {\bf e}_r 
+ F_{3 \theta} {\bf e}_3\otimes {\bf e}_\theta 
+ F_{3 3} {\bf e}_3\otimes {\bf e}_3 
\end{align*}
is independent of $z$, then we have 
\begin{align}\label{def.divF}
\opdiv  F 
= (\opdiv  F)_r {\bf e}_r 
+ (\opdiv  F)_\theta {\bf e}_\theta 
+ (\opdiv  F)_3 {\bf e}_3, 
\end{align}
where
\begin{align*}
(\opdiv  F)_r 
&= 
\frac{1}{r} \partial_r (r F_{r r}) 
+ \frac{1}{r} (\partial_\theta F_{\theta r} - F_{\theta \theta}), \\
(\opdiv  F)_\theta 
&= 
\frac{1}{r} \partial_r (r F_{r \theta}) 
+ \frac{1}{r} (\partial_\theta F_{\theta \theta} + F_{\theta r}), \\
(\opdiv  F)_3 
&= 
\frac{1}{r} \partial_r (r F_{r 3}) 
+ \frac{1}{r} \partial_\theta F_{\theta 3}. 
\end{align*}
In particular, using $\opdiv_\h$ in \eqref{def.op.h.2}, we see that 
\begin{align}\label{rep.divF.3}
(\opdiv F)_3
= \opdiv_\h (F_{r 3} {\bf e}_r + F_{\theta 3} {\bf e}_\theta). 
\end{align}
For a $2\times2$ tensor field $F$ on $\Omega$ 
\begin{align*}
F 
= 
F_{r r} {\bf e}_r\otimes {\bf e}_r 
+ F_{r \theta} {\bf e}_r\otimes {\bf e}_\theta
+ F_{\theta r} {\bf e}_\theta\otimes {\bf e}_r 
+ F_{\theta \theta} {\bf e}_\theta\otimes {\bf e}_\theta, 
\end{align*}
we have 
\begin{align*}
\opdiv_\h F 
= (\opdiv_\h  F)_r {\bf e}_r 
+ (\opdiv_\h  F)_\theta {\bf e}_\theta, 
\end{align*}
where
\begin{align*}
\begin{split}
(\opdiv_\h F)_r 
&= 
\frac{1}{r} \partial_r (r F_{r r}) 
+ \frac{1}{r} (\partial_\theta F_{\theta r} - F_{\theta \theta}), \\
(\opdiv_\h F)_\theta 
&= 
\frac{1}{r} \partial_r (r F_{r \theta}) 
+ \frac{1}{r} (\partial_\theta F_{\theta \theta} + F_{\theta r}). 
\end{split}
\end{align*}

Let $n\in\Z$. For a scalar function $s=s(r,\theta)$ on $\Omega$, we define the projection 
\begin{align}\label{def.proj}
s_n(r) 
= \frac{1}{2\pi} \int_{0}^{2\pi} s(r,t) e^{-int} \dd t 
\end{align}
on the Fourier mode $n$. For a vector field $v=v(r,\theta)$ on $\Omega$ 
\[
v(r,\theta) 
= v_r(r,\theta) {\bf e}_r 
+ v_\theta(r,\theta) {\bf e}_\theta 
+ v_3(r,\theta) {\bf e}_3, 
\]
using \eqref{def.proj}, we define the operator ${\mathcal P}_n$ by 
\begin{align}\label{def.P}
({\mathcal P}_n v)(r,\theta) 
= v_{r,n}(r) e^{in\theta} {\bf e}_r 
+ v_{\theta,n}(r) e^{in\theta} {\bf e}_\theta 
+ v_{3,n}(r) e^{in\theta} {\bf e}_3. 
\end{align}
Applying the operator ${\mathcal P}_n$ to \eqref{def.divF}, we obtain 
\begin{align*}
{\mathcal P}_n (\opdiv  F) 
= 
(\opdiv  F)_{r,n}(r) e^{i n \theta} {\bf e}_r 
+ (\opdiv  F)_{\theta,n}(r) e^{i n \theta} {\bf e}_\theta 
+ (\opdiv  F)_{3,n}(r) e^{i n \theta} {\bf e}_3, 
\end{align*}
where 
\begin{align*}
(\opdiv F)_{r,n}
&= 
\frac{1}{r} \frac{\dd}{\dd r} (r F_{r r, n}) 
+ \frac{1}{r} 
(in F_{\theta r, n} - F_{\theta \theta, n}), \\
(\opdiv  F)_{\theta,n}
&= 
\frac{1}{r} \frac{\dd}{\dd r} (r F_{r \theta, n}) 
+ \frac{1}{r} 
(in F_{\theta \theta, n} + F_{\theta r, n}), \\
(\opdiv  F)_{3,n}
&= 
\frac{1}{r} \frac{\dd}{\dd r} (r F_{r 3, n}) 
+ \frac{in}{r} F_{\theta 3, n}. 
\end{align*}

We also define ${\mathcal P}_n$ acting on scalar- and tensor-valued mappings in an obvious manner. For vector-valued $f$ or tensor-valued $F$, we simply denote ${\mathcal P}_n f$ by $f_n$ or ${\mathcal P}_n F$ by $F_n$. We do not identify ${\mathcal P}_n s$ with $s_n$ when $s$ is scalar-valued to avoid confusion with \eqref{def.proj}.

\subsection{Horizontal Components}
\label{subsec.horiz}

Let $n\in\Z$. Applying ${\mathcal P}_n$ to $v_\h$ and $f_\h$
\begin{align*}
({\mathcal P}_n v_\h)(r,\theta)
&= v_{r,n}(r) e^{in\theta} {\bf e}_r 
+ v_{\theta,n}(r) e^{in\theta} {\bf e}_\theta, \\
({\mathcal P}_n f_\h)(r,\theta)
&= f_{r,n}(r) e^{in\theta} {\bf e}_r 
+ f_{\theta,n}(r) e^{in\theta} {\bf e}_\theta 
\end{align*}
and inserting these to \eqref{eq.LPh}, we see that $(v_{r,n}(r),v_{\theta,n}(r))$ and $q_n(r)$ satisfy 
\begin{align}
&\begin{aligned}\label{eq.LPh.polar.vr}
&-\frac{\dd}{\dd r} \Big(\frac1r\frac{\dd}{\dd r} (r v_{r,n})\Big)
+ \frac{n^2}{r^2} v_{r,n}
+ \frac{2in}{r^2} v_{\theta,n} \\
&\qquad\qquad
- \frac{\alpha}{r^2} \Big(\frac{\dd}{\dd r} (r v_{\theta,n}) - in v_{r,n}\Big)
+ \partial_r q_n 
= f_{r,n}, \quad r>1, \\
\end{aligned} \\
&\begin{aligned}\label{eq.LPh.polar.vtheta}
&-\frac{\dd}{\dd r} \Big(\frac1r\frac{\dd}{\dd r} (r v_{\theta,n})\Big) 
+ \frac{n^2}{r^2} v_{\theta,n}
- \frac{2in}{r^2} v_{r,n} \\
&\qquad\qquad
- \frac{\gamma}{r^2} \Big(\frac{\dd}{\dd r} (r v_{\theta,n}) - in v_{r,n}\Big)
+ \frac{i n}{r} q_n 
= f_{\theta,n}, \quad r>1, 
\end{aligned}
\end{align}
the divergence-free and boundary conditions 
\begin{align}\label{eq.LPh.polar.div-free.bdry}
\frac{\dd}{\dd r}(r v_{r,n}) + in v_{\theta,n}=0, 
\qquad v_{r,n}(1)=v_{\theta,n}(1)=0, 
\end{align}
and the condition at infinity
\begin{align}\label{eq.LPh.polar.spatial.infinity}
|v_{r,n}(r)| + |v_{\theta,n}(r)| \to 0, \quad r\to\infty. 
\end{align}
%

\subsubsection{Axisymmetric Part}
\label{subsubsec.axisym.horiz}

%
\begin{proposition}\label{prop.LPh.zero}
Let $\alpha\in\R$, $\gamma>2$ and $2<\rho\le\gamma$. Suppose that $f_\h=f_{\h,0}$ is given by $f_{\h,0}=\opdiv_\h F_0$ for some $F_0\in {\mathcal P}_0 L^\infty_{2(\rho-1)}(\Omega)^{2\times2}$. Then there is a unique weak solution $v_{\h,0}\in {\mathcal P}_0 L^2_\sigma(\Omega) \cap W^{1,2}_0(\Omega)^2$ of \eqref{eq.LPh} satisfying 
\[
v_{\h,0}(r,\theta)=v_{\theta,0}(r) {\bf e}_\theta
\]
and 
\begin{align}\label{est1.prop.LPh.zero}
\|v_{\h,0}\|_{L^\infty_{\rho-1}} 
+ \frac{1}{\gamma-1} 
\|\nabla_\h v_{\h,0}\|_{L^\infty_{\rho}} 
&\le 
\frac{C(\gamma-1)}{(\gamma-2)(\rho-2)} 
\|F_0\|_{L^\infty_{2(\rho-1)}}. 
\end{align}
The constant $C$ is independent of $\alpha$, $\gamma$ and $\rho$.
\end{proposition}
%

%
\begin{remark}\label{rem.zero.mode.gamma}
The restriction $\gamma>2$ in Theorem \ref{thm.main} can already be seen from Proposition \ref{prop.LPh.zero}. The homogeneous part of \eqref{eq1.proof.prop.LPh.zero} has two independent solutions, $r^{-1}$ and $r^{-\gamma+1}$. Since the theorem fixes this leading profile at infinity, the solution $v_{\h,0}$ must belong to the perturbation class of $O(r^{-\rho+1})$. The estimate $r^{-\gamma+1}=O(r^{-\rho+1})$ is equivalent to $\rho\le\gamma$. Hence the choice $\rho>2$ forces $\gamma>2$. The non-zero modes are governed by a different condition and can decay faster under weaker assumptions, but this does not remove the above zero-mode obstruction for the present normalization at infinity.
\end{remark}
%

%
\begin{proofx}{Proposition \ref{prop.LPh.zero}}
To simplify, we omit the subscript ``$\h$" for $v_\h$ and $f_\h$. Put $n=0$ in \eqref{eq.LPh.polar.vr}--\eqref{eq.LPh.polar.spatial.infinity}. The radial part $v_{r,0}(r)$ is identically zero since both $(r v_{r,0})'=0$ and $v_{r,0}(1)=0$ hold by \eqref{eq.LPh.polar.div-free.bdry}. Besides, the angular part $v_{\theta,0}(r)$ solves the ordinary differential equation 
\begin{align}\label{eq1.proof.prop.LPh.zero}
-\frac{\dd^2 v_{\theta,0}}{\dd r^2} 
-\frac{1+\gamma}r \frac{\dd v_{\theta,0}}{\dd r}  
+\frac{1-\gamma}{r^2} v_{\theta,0}
= f_{\theta,0}, \quad r>1, 
\qquad 
v_{\theta,0}(1)=0. 
\end{align}
As the uniqueness of solutions is clear, we may only consider the existence and estimates.

At first we assume that $F_0\in {\mathcal P}_0 C^\infty_0(\Omega)^{2\times2}$, which gives 
\begin{align}\label{eq2.proof.prop.LPh.zero}
f_{\theta,0}
= (\opdiv_\h F)_{\theta,0}
= \frac{1}{r} \frac{\dd}{\dd r} (r F_{r \theta,0}) 
+ \frac{1}{r} F_{\theta r,0}. 
\end{align}
As is shown in \cite[Proof of Proposition 3.1]{Higaki2023}, the solution of \eqref{eq1.proof.prop.LPh.zero} is represented by 
\begin{equation*}
\begin{split}
v_{\theta,0}(r)
&= 
\frac{1}{\gamma-2}
\bigg\{	
-\bigg(
\int_{1}^{\infty} s^2 f_{\theta,0}(s)\dd s 
\bigg) 
r^{-\gamma+1} \\
&\qquad\qquad\quad
+ r^{-\gamma+1} \int_{1}^{r} s^{\gamma} f_{\theta,0}(s)\dd s 
+ r^{-1} \int_{r}^{\infty} s^2 f_{\theta,0}(s)\dd s
\bigg\}. 
\end{split}
\end{equation*}
By integration by parts based on \eqref{eq2.proof.prop.LPh.zero}, we rewrite this formula as 
\begin{equation}\label{rep1.proof.prop.LPh.zero}
\begin{split}
&v_{\theta,0}(r) \\
&=\frac{1}{\gamma-2}
\bigg[
\bigg(
\int_{1}^{\infty} s F_{r \theta,0}(s) \dd s 
- \int_{1}^{\infty} s F_{\theta r,0}(s) \dd s 
\bigg) 
r^{-\gamma+1} \\
&\qquad\qquad\quad
+ r^{-\gamma+1}
\bigg\{
-(\gamma-1) \int_{1}^{r} s^{\gamma-1} F_{r \theta,0}(s) \dd s 
+ \int_{1}^{r} s^{\gamma-1} F_{\theta r,0}(s) \dd s 
\bigg\} \\
&\qquad\qquad\quad
+ r^{-1} 
\bigg(
-\int_{r}^{\infty} s F_{r \theta,0}(s) \dd s 
+ \int_{r}^{\infty} s F_{\theta r,0}(s) \dd s 
\bigg)
\bigg]. 
\end{split}
\end{equation}
An associated pressure ${\mathcal P}_0 q$ is obtained from the equation \eqref{eq.LPh.polar.vr}. Set $v_0(r,\theta)=v_{\theta,0}(r) {\bf e}_\theta$. Then the pair $(v_0, \nabla_\h{\mathcal P}_0 q)$ is smooth and satisfies \eqref{eq.LPh.polar.vr}--\eqref{eq.LPh.polar.spatial.infinity} in the classical sense.

Next we let $F_0\in {\mathcal P}_0 L^\infty_{2(\rho-1)}(\Omega)^{2\times2}$. By direct computation 
\begin{align*}
\begin{split}
\bigg(\int_{1}^{\infty} 
s |F_0(s)| \dd s 
\bigg) 
r^{-\gamma+1} 
\le 
\frac{C}{\rho-2} 
\|F_0\|_{L^\infty_{2(\rho-1)}} 
r^{-\gamma+1} 
\end{split}
\end{align*}
and 
\begin{align*}
\begin{split}
&r^{-\gamma+1} 
\int_1^r 
s^{\gamma-1} |F_0(s)| \dd s 
+ r^{-1} 
\int_{r}^{\infty}
s |F_0(s)| \dd s 
\le 
\frac{C}{\rho-2} 
\|F_0\|_{L^\infty_{2(\rho-1)}} 
r^{-\rho+1}, 
\end{split}
\end{align*}
one can verify that the desired estimate 
\eqref{est1.prop.LPh.zero} follows from the formula \eqref{rep1.proof.prop.LPh.zero}.

Finally, we show that $v_0(r,\theta)=v_{\theta,0}(r) {\bf e}_\theta$ with $v_{\theta,0}(r)$ defined in \eqref{rep1.proof.prop.LPh.zero} is a weak solution of \eqref{eq.LPh} for general $F_0\in {\mathcal P}_0 L^\infty_{2(\rho-1)}(\Omega)^{2\times2}$. By the definition of $v_0$, we have $v_0\in L^2_{\sigma}(\Omega)\cap W^{1,2}_0(\Omega)^2$. Let $\{F^{(m)}_0\}_{m=1}^\infty\subset {\mathcal P}_0 C^\infty_0(\Omega)^{2\times2}$ be such that $\displaystyle{\lim_{m\to\infty} F^{(m)}_0 = F_0}$ in $L^1(\Omega)\cap L^2(\Omega)$. Let $v^{(m)}_0$ be given by \eqref{rep1.proof.prop.LPh.zero} replacing $F_0$ by $F^{(m)}_0$, and let ${\mathcal P}_0 q^{(m)}$ be the associated pressure. Then, using the estimate for $G\in L^1(0,\infty; s\dd s)^{2\times2}$ 
\begin{align*}
r^{-\gamma} 
\int_1^r 
s^{\gamma-1} |G(s)| \dd s 
+ r^{-2} 
\int_{r}^{\infty}
s |G(s)| \dd s 
\le
\bigg(\int_{1}^{\infty} |G(s)| s\dd s\bigg) r^{-2}
\end{align*}
which implies 
\[
\|\nabla_\h (v_0-v^{(m)}_0)\|_{L^2} 
\le 
C(\|F_0-F^{(m)}_0\|_{L^1} + \|F_0-F^{(m)}_0\|_{L^2}),
\]
we have, by integration by parts, 
\begin{align*}
\begin{split}
&\langle \nabla_\h v_0, \nabla_\h \varphi \rangle 
+ \langle V_\h^\bot \oprot_{{\rm 2D}} v_0, \varphi \rangle 
+ \langle F_0, \nabla_\h \varphi \rangle \\
&= 
\lim_{m\to\infty}
\Big(
\langle \nabla_\h v^{(m)}_0, \nabla_\h \varphi \rangle
+ \langle V_\h^\bot \oprot_{{\rm 2D}} v^{(m)}_0, \varphi \rangle
+ \langle F^{(m)}_0, \nabla_\h \varphi \rangle
\Big) \\
&= \lim_{m\to\infty} 
\langle 
-\Delta_\h v^{(m)}_0 
+ V_\h^\bot \oprot_{{\rm 2D}} v^{(m)}_0 
- \opdiv_\h F^{(m)}_0, \varphi \rangle \\
&= \lim_{m\to\infty} 
\langle \nabla_\h {\mathcal P}_0 q^{(m)}, \varphi \rangle = 0, 
\quad 
\varphi\in C^\infty_{0,\sigma}(\Omega). 
\end{split}
\end{align*}
Hence we see that $v_0$ is a weak solution of \eqref{eq.LPh}. This completes the proof. 
\end{proofx}
%

\subsubsection{Non-axisymmetric Part}
\label{subsubsec.nonaxisym.horiz}

Let $n\neq0$. We set 
\begin{align}\label{def.zeta}
n_\gamma = \Big\{n^2 + \Big(\frac{\gamma}{2}\Big)^2\Big\}^{\frac12}, \qquad 
\zeta_n = (n_\gamma^2 + i\alpha n)^{\frac12} 
\end{align}
and 
\begin{align}\label{def.xi}
\begin{split}
\xi_n
&= 
\Re(\zeta_n) 
= 
\frac{n_\gamma}{\sqrt{2}} 
\bigg[\Big\{1 + \Big(\frac{\alpha n}{n_\gamma^2}\Big)^2\Big\}^{\frac12} + 1 \bigg]^{\frac12}. 
\end{split}
\end{align}
Here the square root is chosen so that $\Re(\zeta_n)>0$. By direct computation, 
\begin{align}\label{ineqs.xi}
\xi_n \le |\zeta_n| \le \sqrt{2} \xi_n, 
\qquad 
C_1
\le
\frac{\xi_n}{|n|}
\le 
(|\alpha|^{\frac12} + \gamma)
, 
\qquad 
0<\Big(\xi_n - \frac{\gamma}2\Big)^{-1}
< C_2\frac{\gamma}{|n|}
\end{align}
with $C_1, C_2$ independent of $n$, $\alpha$ and $\gamma$.

%
\begin{proposition}\label{prop.LPh.nonzero}
Let $n\neq0$ and let $\alpha\in\R$, $\gamma>2$ and $2<\rho<3$ with $\rho\le\gamma$. Suppose that $f_\h=f_{\h,n}$ is given by $f_{\h,n}=\opdiv_\h F_n$ for some $F_n\in {\mathcal P}_n L^\infty_{2(\rho-1)}(\Omega)^{2\times2}$. Then there is a unique weak solution $v_{\h,n}\in {\mathcal P}_n L^2_\sigma(\Omega) \cap W^{1,2}_0(\Omega)^2$ of \eqref{eq.LPh} satisfying 
\begin{align}\label{est1.prop.LPh.nonzero}
\|v_{\h,n}\|_{L^\infty_{\rho-1}} 
+ \frac{1}{|n|} \|\nabla_\h v_{\h,n}\|_{L^\infty_{\rho}} 
\le 
\frac{C}{|n|(|n|-\rho+2)} 
\xi_n^2 
\Big(\xi_n - \frac{\gamma}2\Big)^{-1} 
\|F_n\|_{L^\infty_{2(\rho-1)}}. 
\end{align}
The constant $C$ is independent of $n$, $\alpha$, $\gamma$ and $\rho$.
\end{proposition}
%
%
\begin{proof}
To simplify, we omit the subscript ``$\h$" for $v_\h$ and $f_\h$. As the uniqueness of solutions can be shown as in \cite[Proof of Proposition 3.2]{Higaki2023}, we only prove the existence and estimates.

At first we assume that $F_n\in {\mathcal P}_n C^\infty_0(\Omega)^{2\times2}$, which gives  
\begin{align}\label{eq1.proof.prop.LPh.nonzero}
\begin{split}
f_{r,n} 
&=
(\opdiv_\h F)_{r,n} 
= 
\frac{1}{r} \frac{\dd}{\dd r} (r F_{r r, n}) 
+ \frac{1}{r} 
(in F_{\theta r, n} - F_{\theta \theta, n}), \\
f_{\theta,n} 
&=
(\opdiv_\h F)_{\theta,n}
= 
\frac{1}{r} \frac{\dd}{\dd r} (r F_{r \theta, n}) 
+ \frac{1}{r} 
(in F_{\theta \theta, n} + F_{\theta r, n}). 
\end{split}
\end{align}
Set
\begin{align*}
\omega_n(r) 
= \big((\oprot_{{\rm 2D}} v_n)e^{-in\theta}\big)(r). 
\end{align*}
From \cite[Proof of Proposition 3.2]{Higaki2023}, we see that $\omega_n(r)$ solves 
\begin{align}\label{eq0.proof.prop.LPh.nonzero}
-\frac{\dd^2 \omega_n}{\dd r^2} 
- \frac{1+\gamma}{r} \frac{\dd \omega_n}{\dd r} 
+ \frac{n^2+i\alpha n}{r^2} \omega_n 
= (\oprot_{{\rm 2D}} f_n)_n, 
\quad r>1 
\end{align}
and, recalling $\zeta_n$ in \eqref{def.zeta}, that it can be represented by 
\begin{align}\label{rep1.proof.prop.LPh.nonzero}
\omega_n(r) 
= \Phi_n[f_n](r) + c_n[f_n] r^{-\zeta_n-\frac{\gamma}2}. 
\end{align}
Here
\footnote{Due to a typographical error, the denominator $2\zeta_n$ in $\Phi_n[f_n](r)$ is written as $\zeta_n$ in \cite[(3.19)--(3.20)]{Higaki2023}. It should be emphasized, however, that this error does not affect the main results in \cite{Higaki2023}.}
\begin{align*}
\begin{split}
\Phi_n[f_n](r)
&= 
\frac{r^{-\zeta_n-\frac{\gamma}2}}{2\zeta_n} 
\int_1^r s^{\zeta_n + \frac{\gamma}{2} + 1} (\oprot_{{\rm 2D}} f_n)_n(s) \dd s \\
&\quad
+ \frac{r^{\zeta_n-\frac{\gamma}2}}{2\zeta_n}
\int_r^\infty s^{-\zeta_n + \frac{\gamma}{2} + 1} (\oprot_{{\rm 2D}} f_n)_n(s) \dd s \\
&= 
-\frac{r^{-\zeta_n-\frac{\gamma}2}}{2\zeta_n} 
\int_1^r s^{\zeta_n + \frac{\gamma}2} 
\Big\{
inf_{r,n}(s)
+ \Big(\zeta_n + \frac{\gamma}2\Big) f_{\theta,n}(s) 
\Big\} \dd s \\
&\quad
+ \frac{r^{\zeta_n-\frac{\gamma}2}}{2\zeta_n}
\int_r^\infty s^{-\zeta_n + \frac{\gamma}2} 
\Big\{
-inf_{r,n}(s)
+ \Big(\zeta_n - \frac{\gamma}2\Big) f_{\theta,n}(s) 
\Big\} \dd s 
\end{split}
\end{align*}
and 
\begin{align}\label{rep2.proof.prop.LPh.nonzero}
\begin{split}
c_n[f_n]
&= 
-\Big(\zeta_n + |n| + \frac{\gamma}2 - 2\Big) 
\int_1^\infty s^{-|n|+1} \Phi_n[f_n](s) \dd s. 
\end{split}
\end{align}
By integration by parts based on \eqref{eq1.proof.prop.LPh.nonzero}, we rewrite the formula of $\Phi_n[f_n](r)$ as 
\begin{align}\label{rep3.proof.prop.LPh.nonzero}
\begin{split}
\Phi_n[f_n](r)
&= 
-F_{r \theta, n}(r) 
+ r^{-\zeta_n-\frac{\gamma}2} 
\int_1^r 
s^{\zeta_n + \frac{\gamma}2-1} 
G_{1}(s) \dd s \\
&\quad
+ r^{\zeta_n-\frac{\gamma}2} 
\int_r^\infty 
s^{-\zeta_n + \frac{\gamma}2-1} 
G_{2}(s) \dd s, 
\end{split}
\end{align}
where 
\begin{align*}
\begin{split}
G_{1}(r)
&= 
\frac{in}{2\zeta_n} 
\Big(\zeta_n + \frac{\gamma}2-1\Big)
F_{r r, n}(r)
+ \frac{1}{2\zeta_n} 
\Big(\zeta_n + \frac{\gamma}2\Big)
\Big(\zeta_n + \frac{\gamma}2-1\Big)
F_{r \theta, n}(r) \\
&\quad 
- \frac{1}{2\zeta_n} 
\Big(\zeta_n + \frac{\gamma}2-n^2\Big)
F_{\theta r, n}(r)
- \frac{in}{2\zeta_n} 
\Big(\zeta_n + \frac{\gamma}2-1\Big)
F_{\theta \theta, n}(r), \\
G_{2}(r)
&= 
-\frac{in}{2\zeta_n} 
\Big(\zeta_n - \frac{\gamma}2 + 1\Big)
F_{r r, n}(r)
+ \frac{1}{2\zeta_n} 
\Big(\zeta_n - \frac{\gamma}2\Big)
\Big(\zeta_n - \frac{\gamma}2 + 1\Big)
F_{r \theta, n}(r) \\
&\quad 
+ \frac{1}{2\zeta_n} 
\Big(\zeta_n - \frac{\gamma}2 + n^2\Big)
F_{\theta r, n}(r)
+ \frac{in}{2\zeta_n} 
\Big(\zeta_n - \frac{\gamma}2 + 1\Big)
F_{\theta \theta, n}(r). 
\end{split}
\end{align*}
Then, using $\omega_n(r)$ in \eqref{rep1.proof.prop.LPh.nonzero}, we see that $v_n$ defined by the Biot-Savart law 
\begin{align}\label{rep4.proof.prop.LPh.nonzero}
\begin{split}
v_n(r,\theta)
&=v_{r,n}(r) e^{in\theta} {\bf e}_r + v_{\theta,n}(r) e^{in\theta} {\bf e}_\theta, \\
v_{r,n}(r)
&=\frac{in}{2|n|}
\Big(r^{-|n|-1} \int_{1}^{r} s^{|n|+1} \omega_n(s)\dd s + r^{|n|-1} \int_{r}^{\infty} s^{-|n|+1} \omega_n(s)\dd s
\Big), \\
v_{\theta,n}(r)
&=\frac{1}{2}
\Big(r^{-|n|-1} \int_{1}^{r} s^{|n|+1} \omega_n(s)\dd s - r^{|n|-1} \int_{r}^{\infty} s^{-|n|+1} \omega_n(s)\dd s
\Big) 
\end{split}
\end{align}
is a smooth solution of \eqref{eq.LPh} with some smooth associated pressure ${\mathcal P}_n q$. The reader is referred to \cite[Proof of Proposition 3.2]{Higaki2023} for the details when $f$ is a regular function.

Next we let $F_n\in {\mathcal P}_n L^\infty_{2(\rho-1)}(\Omega)^{2\times2}$. Recalling \eqref{def.xi}--\eqref{ineqs.xi}, we compute 
\begin{align*}
\begin{split}
&\bigg|
r^{-\zeta_n-\frac{\gamma}2} 
\int_1^r 
s^{\zeta_n + \frac{\gamma}2-1} 
G_{1}(s) \dd s
\bigg|
+ 
\bigg|
r^{\zeta_n-\frac{\gamma}2} 
\int_r^\infty 
s^{-\zeta_n + \frac{\gamma}2-1} 
G_{2}(s) \dd s 
\bigg| \\
&\le
C \xi_n
\bigg(
r^{-\xi_n - \frac{\gamma}2} 
\int_1^r 
s^{\xi_n+\frac{\gamma}2-2\rho+1} \dd s
+ 
r^{\xi_n-\frac{\gamma}2} 
\int_r^\infty 
s^{-\xi_n+\frac{\gamma}2-2\rho+1} \dd s
\bigg) \\
&\le 
C \xi_n
\Big(\xi_n-\frac{\gamma}{2}\Big)^{-1} 
\|F_n\|_{L^\infty_{2(\rho-1)}} 
r^{-\rho} 
\end{split}
\end{align*}
to estimate the terms in \eqref{rep1.proof.prop.LPh.nonzero}--\eqref{rep3.proof.prop.LPh.nonzero} as
\begin{align*}
\begin{split}
r^{\rho} |\Phi_n[f_n](r)| 
+ \frac{|n|}{\xi_n} |c_n[f_n]| 
+ \frac{|n|}{\xi_n} r^{\rho} |\omega_n(r)|
\le 
C \xi_n\Big(\xi_n-\frac{\gamma}{2}\Big)^{-1} 
\|F_n\|_{L^\infty_{2(\rho-1)}}. 
\end{split}
\end{align*}
Hence the desired estimate \eqref{est1.prop.LPh.nonzero} is obtained by \eqref{rep4.proof.prop.LPh.nonzero} and easy computation using 
\begin{align*}
\begin{split}
&|v_n(r,\theta)| 
+ \frac{r}{|n|} |\nabla v_n(r,\theta)| \\
&\le 
\frac{C}{|n|} 
\xi_n^2 
\Big(\xi_n - \frac{\gamma}2\Big)^{-1} 
\|F_n\|_{L^\infty_{2(\rho-1)}} \\
&\quad
\times
\bigg(r^{-|n|-1} \int_{1}^{r} s^{|n|-\rho+1} \dd s 
+ r^{|n|-1} \int_{r}^{\infty} s^{-|n|-\rho+1} \dd s
\bigg). 
\end{split}
\end{align*}

It remains to show that $v_n$ defined in \eqref{rep4.proof.prop.LPh.nonzero} is a weak solution of \eqref{eq.LPh} for general $F_n\in {\mathcal P}_n L^\infty_{2(\rho-1)}(\Omega)^{2\times2}$. However, the proof by a density argument using the estimates 
\begin{align*}
r^{-\xi_n-\frac{\gamma}{2}}
\int_{1}^{r} s^{\xi_n+\frac{\gamma}{2}-1} |H(s)| \dd s
&\le
\bigg(
\int_{1}^{\infty} |H(s)| s \dd s 
\bigg) 
r^{-2}, \\
r^{\xi_n-\frac{\gamma}{2}}
\int_{r}^{\infty} s^{-\xi_n+\frac{\gamma}2-1} |H(s)| \dd s 
&\le
\bigg(
\int_{1}^{\infty} |H(s)| s \dd s 
\bigg) 
r^{-2} 
\end{align*}
is similar to the one of Proposition \ref{prop.LPh.zero}. Thus we omit the details. The proof is complete. 
\end{proof}
%

\subsection{Vertical Component}
\label{subsec.vert}

Let $n\in\Z$. Applying ${\mathcal P}_n$ to $v_3$ and $f_3$
\begin{align*}
({\mathcal P}_n v_3)(r,\theta)
= v_{3,n}(r) e^{in\theta}, 
\qquad 
({\mathcal P}_n f_3)(r,\theta)
= f_{3,n}(r) e^{in\theta} 
\end{align*}
and inserting these to \eqref{eq.LPv}, we see that $v_{3,n}(r)$ satisfies 
\begin{align}\label{eq.LPz.polar}
-\frac{\dd^2 v_{3,n}}{\dd r^2} 
- \frac{1+\gamma}{r} \frac{\dd v_{3,n}}{\dd r} 
+ \frac{n^2+i\alpha n}{r^2} v_{3,n} 
= f_{3,n}, 
\quad r>1 
\end{align}
and the boundary conditions 
\begin{align}\label{eq.LPz.polar.BCs}
v_{3,n}(1) = 0, 
\qquad
|v_{3,n}(r)| \to 0, \quad r\to\infty. 
\end{align}
%

\subsubsection{Axisymmetric Part}
\label{subsubsec.axisym.vert}

%
\begin{proposition}\label{prop.LPz.zero}
Let $\alpha\in\R$, $\gamma>2$ and $2<\rho\le\gamma$. 
\begin{enumerate}[(1)]
\item\label{item1.prop.LPz.zero}
Suppose that $f_3={\mathcal P}_0 f_3\in {\mathcal P}_0 L^\infty_{2\rho-1}(\Omega)$. Then there is a unique weak solution ${\mathcal P}_0 v_3\in {\mathcal P}_0 W^{1,2}_0(\Omega)$ of \eqref{eq.LPv} satisfying 
\begin{align}\label{est1.prop.LPz.zero}
\|{\mathcal P}_0 v_3\|_{L^\infty_{\rho-1}} 
+ \frac{1}{\gamma} \|\nabla_\h {\mathcal P}_0 v_3\|_{L^\infty_{\rho}} 
\le 
\frac{C}{\gamma(\rho-2)}
\|{\mathcal P}_0 f_3\|_{L^\infty_{2\rho-1}}. 
\end{align}

\item\label{item2.prop.LPz.zero}
Suppose that $f_3={\mathcal P}_0 f_3$ is given by ${\mathcal P}_0 f_3=\opdiv_\h F_0$ for some $F_0\in {\mathcal P}_0 L^\infty_{2(\rho-1)}(\Omega)^2$. Then there is a unique weak solution ${\mathcal P}_0 v_3\in {\mathcal P}_0 W^{1,2}_0(\Omega)$ of \eqref{eq.LPv} satisfying 
\begin{align}\label{est2.prop.LPz.zero}
\|{\mathcal P}_0 v_3\|_{L^\infty_{\rho-1}} 
+ \frac{1}{\gamma} \|\nabla_\h {\mathcal P}_0 v_3\|_{L^\infty_{\rho}} 
\le 
\frac{C}{\rho-2} \|F_0\|_{L^\infty_{2(\rho-1)}}. 
\end{align}
\end{enumerate}
The constant $C$ is independent of $\alpha$, $\gamma$ and $\rho$.
\end{proposition}
%
%
\begin{proof}
We may only consider the existence and estimates of solutions.

(\ref{item1.prop.LPz.zero}) 
Putting $n=0$ in \eqref{eq.LPz.polar}--\eqref{eq.LPz.polar.BCs}, we see that $v_{3,0}(r)$ solves 
\begin{align}\label{eq1.proof.prop.LPz.zero}
-\frac{\dd^2 v_{3,0}}{\dd r^2} 
- \frac{1+\gamma}{r} \frac{\dd v_{3,0}}{\dd r} 
= f_{3,0}, 
\quad r>1, 
\qquad 
v_{3,0}(1) = 0. 
\end{align}
The linearly independent solutions of the homogeneous equation of \eqref{eq1.proof.prop.LPz.zero} are
\begin{align*}
r^{-\gamma} 
\quad {\rm and} \quad 
1,
\end{align*}
and their Wronskian is $\gamma r^{-\gamma-1}$. Hence the solution of \eqref{eq1.proof.prop.LPz.zero} vanishing at infinity is 
\begin{align}\label{rep1.proof.prop.LPz.zero}
\begin{split}
v_{3,0}(r) 
= 
\frac{1}{\gamma}
\bigg\{	
&-\bigg(
\int_{1}^{\infty} s f_{3,0}(s)\dd s 
\bigg) 
r^{-\gamma} \\
&
+ r^{-\gamma} \int_{1}^{r} s^{\gamma+1} f_{3,0}(s)\dd s 
+ \int_{r}^{\infty} s f_{3,0}(s)\dd s
\bigg\}. 
\end{split}
\end{align}
Thus the estimate \eqref{est1.prop.LPz.zero} follows from 
\begin{align}\label{est1.proof.prop.LPz.zero}
\begin{split}
r^{-\gamma} 
\int_{1}^{r} s^{\gamma-2\rho+2} \dd s 
+ \int_{r}^{\infty} s^{-2\rho+2} \dd s
\le 
\frac{C}{\rho-2} 
r^{-\rho+1}. 
\end{split}
\end{align}
One easily checks that $v_{3,0}(r)$ defined in \eqref{rep1.proof.prop.LPz.zero} is a weak solution of \eqref{eq.LPv}.

(\ref{item2.prop.LPz.zero}) 
At first we assume that $F_0\in {\mathcal P}_0 C^\infty_0(\Omega)^2$, which gives 
\begin{align*}
f_{3,0}
= \opdiv_\h F_0
= \frac{1}{r} \frac{\dd}{\dd r} (r F_{r,0}). 
\end{align*}
By integration by parts, we rewrite the formula \eqref{rep1.proof.prop.LPz.zero} as 
\begin{equation}\label{rep2.proof.prop.LPz.zero}
\begin{split}
v_{3,0}(r) 
= -r^{-\gamma} \int_{1}^{r} s^{\gamma} F_{r,0}(s)\dd s. 
\end{split}
\end{equation}

Next we let $F_0\in {\mathcal P}_0 L^\infty_{2(\rho-1)}(\Omega)^2$. The estimate \eqref{est2.prop.LPz.zero} follows from \eqref{est1.proof.prop.LPz.zero}. One can verify that $v_{3,0}(r)$ defined in \eqref{rep2.proof.prop.LPz.zero} is a weak solution of \eqref{eq.LPv} by a density argument similar to the one in the proof of Proposition \ref{prop.LPh.zero}. This completes the proof. 
\end{proof}
%

\subsubsection{Non-axisymmetric Part}
\label{subsubsec.nonaxisym.vert}

Recall that $\xi_n$ is defined in \eqref{def.xi}. 
%
\begin{proposition}\label{prop.LPz.nonzero}
Let $n\neq0$ and let $\alpha\in\R$, $\gamma>2$ and $2<\rho<3$ with $\rho\le\gamma$.

\begin{enumerate}[(1)]
\item\label{item1.prop.LPz.nonzero}
Suppose that $f_3={\mathcal P}_n f_3\in {\mathcal P}_n L^\infty_{2\rho-1}(\Omega)$. Then there is a unique weak solution ${\mathcal P}_n v_3\in {\mathcal P}_n W^{1,2}_0(\Omega)$ of \eqref{eq.LPv} satisfying 
\begin{align}\label{est1.prop.LPz.nonzero}
\|{\mathcal P}_n v_3\|_{L^\infty_{\rho-1}} 
+ \frac{1}{\xi_n} \|\nabla_\h {\mathcal P}_n v_3\|_{L^\infty_{\rho}} 
\le 
\frac{C}{\xi_n} \Big(\xi_n - \frac{\gamma}2\Big)^{-1} 
\|{\mathcal P}_n f_3\|_{L^\infty_{2\rho-1}}. 
\end{align}

\item\label{item2.prop.LPz.nonzero}
Suppose that $f_3={\mathcal P}_n f_3$ is given by ${\mathcal P}_n f_3=\opdiv_\h F_n$ for some $F_n\in {\mathcal P}_n L^\infty_{2(\rho-1)}(\Omega)^2$. Then there is a unique weak solution ${\mathcal P}_n v_3\in {\mathcal P}_n W^{1,2}_0(\Omega)$ of \eqref{eq.LPv} satisfying 
\begin{align}\label{est2.prop.LPz.nonzero}
\|{\mathcal P}_n v_3\|_{L^\infty_{\rho-1}} 
+ \frac{1}{\xi_n} \|\nabla_\h {\mathcal P}_n v_3\|_{L^\infty_{\rho}} 
\le 
C \Big(\xi_n - \frac{\gamma}2\Big)^{-1} 
\|F_n\|_{L^\infty_{2(\rho-1)}}. 
\end{align}
\end{enumerate}
The constant $C$ is independent of $n$, $\alpha$, $\gamma$ and $\rho$.
\end{proposition}
%
%
\begin{proof}
We may only consider the existence and estimates of solutions.

(\ref{item1.prop.LPz.nonzero}) 
Since $v_{3,n}(r)$ satisfies \eqref{eq0.proof.prop.LPh.nonzero} with $(\oprot_{{\rm 2D}} f_n)_n$ replaced by $f_{3,n}$, it is given by 
\begin{align}\label{rep1.proof.prop.LPz.nonzero}
\begin{split}
&v_{3,n}(r) \\
&= 
\frac{1}{2\zeta_n}
\bigg\{	
-\bigg(
\int_1^\infty s^{-\zeta_n + \frac{\gamma}{2} + 1} f_{3,n}(s) \dd s 
\bigg) 
r^{-\zeta_n-\frac{\gamma}2} \\
&\qquad\qquad
+ r^{-\zeta_n-\frac{\gamma}2} 
\int_1^r s^{\zeta_n + \frac{\gamma}{2} + 1} f_{3,n}(s) \dd s 
+ r^{\zeta_n-\frac{\gamma}2} 
\int_r^\infty s^{-\zeta_n + \frac{\gamma}{2} + 1} f_{3,n}(s) \dd s 
\bigg\}. 
\end{split}
\end{align}
Thus the estimate \eqref{est1.prop.LPz.nonzero} follows from \eqref{def.xi}--\eqref{ineqs.xi} and 
\begin{align}\label{est1.proof.prop.LPz.nonzero}
\begin{split}
&r^{-\xi_n-\frac{\gamma}2} 
\int_1^r 
s^{\xi_n+\frac{\gamma}2-2\rho+2} 
\dd s
+ r^{\xi_n-\frac{\gamma}2} 
\int_r^\infty 
s^{-\xi_n+\frac{\gamma}2-2\rho+2} 
\dd s \\
&\le 
C \Big(\xi_n - \frac{\gamma}2\Big)^{-1} 
r^{-\rho+1}. 
\end{split}
\end{align}
One easily checks that $v_{3,n}(r) e^{in\theta}$ with $v_{3,n}(r)$ defined in \eqref{rep1.proof.prop.LPz.nonzero} is a weak solution of \eqref{eq.LPv}.

(\ref{item2.prop.LPz.nonzero}) 
At first we assume that $F_n\in {\mathcal P}_n C^\infty_0(\Omega)^2$, which gives
\begin{align*}
f_{3,n}
= 
(\opdiv_\h F)_n
= 
\frac{1}{r} 
\frac{\dd}{\dd r} (r F_{r,n}) 
+ \frac{in}{r} F_{\theta,n}. 
\end{align*}
By integration by parts, we rewrite the formula \eqref{rep1.proof.prop.LPz.nonzero} as 
\begin{equation}\label{rep2.proof.prop.LPz.nonzero}
\begin{split}
v_{3,n}(r)
&=
\frac{1}{2\zeta_n}
\bigg[
-\bigg\{
\int_1^\infty 
s^{-\zeta_n + \frac{\gamma}{2}} 
\bigg(
\Big(\zeta_n - \frac{\gamma}2\Big)
F_{r,n}(s)
+ in F_{\theta,n}(s)
\bigg)
\dd s
\bigg\} 
r^{-\zeta_n-\frac{\gamma}2} \\
&\qquad\qquad\quad
+r^{-\zeta_n-\frac{\gamma}2} 
\int_1^r s^{\zeta_n + \frac{\gamma}{2}} 
\bigg(
-\Big(\zeta_n + \frac{\gamma}2\Big)
F_{r,n}(s)
+ in F_{\theta,n}(s)
\bigg)
\dd s \\
&\qquad\qquad\quad
+ r^{\zeta_n-\frac{\gamma}2} 
\int_r^\infty 
s^{-\zeta_n + \frac{\gamma}{2}} 
\bigg(
\Big(\zeta_n - \frac{\gamma}2\Big)
F_{r,n}(s)
+ in F_{\theta,n}(s)
\bigg)
\dd s
\bigg]. 
\end{split}
\end{equation}

Next we let $F_n\in {\mathcal P}_n L^\infty_{2(\rho-1)}(\Omega)^2$. The estimate \eqref{est2.prop.LPz.nonzero} follows from \eqref{est1.proof.prop.LPz.nonzero}. One can check that $v_{3,n}(r) e^{in\theta}$ with $v_{3,n}(r)$ defined in \eqref{rep2.proof.prop.LPz.nonzero} is a weak solution of \eqref{eq.LPv} by a density argument similar to the one in the proof of Proposition \ref{prop.LPh.zero}. This completes the proof. 
\end{proof}
%

\section{Proof of Theorem \ref{thm.main}}
\label{sec.pf}

We prove Theorem \ref{thm.main} in this section. As the proof is similar to \cite[Proof of Theorem 1.1]{Higaki2023}, we give only the outline to avoid duplication. For $\rho\ge0$, we define the Banach space
\begin{align*}
l^1\big(L^\infty_{\rho}(\Omega)\big)
&= \bigg\{f=\sum_{n\in\Z}\mathcal{P}_n f~\bigg|~ 
\|f\|_{l^1L^\infty_\rho}:=\sum_{n\in\Z} \|\mathcal{P}_n f\|_{L^\infty_\rho}<\infty\bigg\}.
\end{align*}
The following is a corollary of Propositions \ref{prop.LPh.zero}--\ref{prop.LPz.nonzero} and the results in \cite{Higaki2023} for 2D problems. 
%
\begin{corollary}\label{cor.est.l1}
Let $\alpha\in\R$, $\gamma>2$ and $2<\rho<3$ with $\rho\le\gamma$. Suppose that the external force $f$ is a distribution on $\Omega$ given by $f = g + \opdiv F$ where $g\in l^1\big(L^\infty_{2\rho-1}(\Omega)\big)^3$ and $F\in l^1\big(L^\infty_{2(\rho-1)}(\Omega)\big)^{3\times3}$. Then there is a unique weak solution $v=(v_\h,v_3)\in (L^2_\sigma(\Omega)\times L^2(\Omega)) \cap W^{1,2}_0(\Omega)^3$ of \eqref{eq.LPh}--\eqref{eq.LPv} satisfying 
\begin{align}\label{est1.cor.est.l1}
\begin{split}
\|v\|_{l^1 L^\infty_{\rho-1}} 
+ \|\nabla_\h v\|_{l^1 L^\infty_{\rho}} 
\le 
\lambda 
\big(
\|g\|_{l^1 L^\infty_{2\rho-1}}
+ \|F\|_{l^1 L^\infty_{2(\rho-1)}}
\big), 
\end{split}
\end{align}
where $\lambda=\lambda(\alpha,\gamma,\rho)$ is given by 
\begin{align*}
\begin{split}
\lambda
= 
\frac{C_0\gamma^2(|\alpha|^\frac12+\gamma)^2}{(\rho-2)^2(3-\rho)}. 
\end{split}
\end{align*}
The constant $C_0$ is independent of $\alpha$, $\gamma$ and $\rho$.
\end{corollary}
%
%
\begin{proof}
By \cite[Lemma 4.1]{Higaki2023}, for solutions of the 2D problem
\begin{equation*}
\left\{
\begin{array}{ll}
-\Delta_\h w_\h + (V_\h)^\bot \oprot_{{\rm 2D}} w_\h + \nabla_\h s 
= g_\h&\mbox{in}\ \Omega \\
\opdiv_\h w_\h = 0&\mbox{in}\ \Omega \\
w_\h = 0 &\mbox{on}\ \partial\Omega \\
w_\h(x)\to0&\mbox{as}\ |x|\to\infty, 
\end{array}\right.
\end{equation*}
there is a constant $\tilde{C_0}$ independent of $\alpha$, $\gamma$ and $\rho$ such that 
\begin{align}\label{est1.proof.cor.est.l1}
\begin{split}
\|w_\h\|_{l^1 L^\infty_{\rho-1}} 
+ \|\nabla_\h w_\h\|_{l^1 L^\infty_{\rho}} 
\le 
\frac{\tilde{C_0}\gamma(|\alpha|^\frac12+\gamma)}{(\rho-2)^2(3-\rho)} 
\|g_\h\|_{l^1 L^\infty_{2\rho-1}}. 
\end{split}
\end{align}
Hence the desired estimate \eqref{est1.cor.est.l1} is a consequence of the estimates in Propositions \ref{prop.LPh.zero}--\ref{prop.LPz.nonzero} combined with \eqref{rep.divF.3}, \eqref{def.xi}--\eqref{ineqs.xi} and \eqref{est1.proof.cor.est.l1}. In addition, the existence and uniqueness of solutions follow from Propositions \ref{prop.LPh.zero}--\ref{prop.LPz.nonzero} and \cite[Lemma 4.1]{Higaki2023}. The proof is complete. 
\end{proof}
%

%
\begin{proofx}{Theorem \ref{thm.main}}
We consider the Banach space 
\begin{align*}
{\mathcal X}_\rho
=\Big\{w\in 
W^{1,2}_0(\Omega)^3 \cap l^1\big(L^\infty_{\rho-1}(\Omega)\big)^3 
~\Big|~ 
w_\h\in L^2_\sigma(\Omega), 
\mkern9mu
\nabla_\h w \in l^1\big(L^\infty_{\rho}(\Omega)\big)^{2\times3} \Big\} 
\end{align*}
equipped with the norm $\|w\|_{{\mathcal X}_\rho}:=\|w\|_{l^1 L^\infty_{\rho-1}}+\|\nabla_\h w\|_{l^1 L^\infty_{\rho}}$. By Corollary \ref{cor.est.l1}, for any $w\in {\mathcal X}_\rho$, there is a unique weak solution $v_w$ to \eqref{eq.LPh}--\eqref{eq.LPv} with the external force 
\[
-w\cdot\nabla w + f 
= g + \opdiv (-w\otimes w + F)
\] 
and the solution $v_w$ satisfies 
\begin{align*}
\begin{split}
\|v_w\|_{l^1 L^\infty_{\rho-1}} 
+ \|\nabla_\h v_w\|_{l^1 L^\infty_{\rho}} 
&\le 
\lambda 
\big(
\|g\|_{l^1 L^\infty_{2\rho-1}}
+ \|-w\otimes w + F\|_{l^1 L^\infty_{2(\rho-1)}}
\big) \\
&\le 
\lambda 
\big(
\|w\|_{{\mathcal X}_\rho}^2
+ \|g\|_{l^1 L^\infty_{2\rho-1}}
+ \|F\|_{l^1 L^\infty_{2(\rho-1)}}
\big). 
\end{split}
\end{align*}
We have used the Young inequality for sequences to estimate $-w\otimes w$ in the last line. Hence the mapping ${\mathcal X}_\rho\ni w\mapsto v_w\in {\mathcal X}_\rho$ is well-defined, and will be denoted by $T$.

It is not hard to check that $T$ is a contraction on the closed subset 
\begin{align*}
\mathcal{B}_{\rho}(\delta) 
= \{w\in {\mathcal X}_\rho ~|~ \|w\|_{{\mathcal X}_\rho} \le \delta \}, 
\quad \delta>0 
\end{align*}
if $g,F,\delta$ are small enough depending on $\lambda=\lambda(\alpha,\gamma,\rho)$. Thus the existence of a weak solution of \eqref{intro.eq.NP} in Introduction unique in $\mathcal{B}_{\rho}(\delta)$ follows from the Banach fixed-point theorem.

Let us set $u\in \widehat{W}^{1,2}(\Omega)^3$ by $u = V + v$. Then one can check that $u$ is a weak solution of \eqref{intro.eq.NS} with $b=(\alpha x^{\perp}-\gamma x,0)$ by using \eqref{intro.eq.bilinear.rot}. Moreover, $u$ is unique in the set
\begin{align*}
\{u ~|~ u = V + v, \mkern9mu v \in \mathcal{B}_{\rho}(\delta) \}  
\end{align*}
and satisfies the limit \eqref{est1.thm.main}. This completes the proof of Theorem \ref{thm.main}. 
\end{proofx}
%

\section*{Acknowledgment}

MH was supported by JSPS KAKENHI Grant Numbers JP 25K17278 and 25K00915.

\end{document}